\documentclass[11pt]{amsart}
\usepackage{amsmath}
\usepackage{amsfonts}
\usepackage{amssymb,enumerate}
\usepackage{amsthm}
\usepackage{hyperref}
\usepackage{centernot}
\usepackage{stmaryrd}

\theoremstyle{plain}
\newtheorem{theorem}{Theorem}[section]
\newtheorem{lemma}[theorem]{Lemma}
\newtheorem{definition}[theorem]{Definition}

\newtheorem{question}[theorem]{Question}

\begin{document}

\title{Set mappings on 4-tuples}
\author{Shahram Mohsenipour}
\author{Saharon Shelah}
\address{School of Mathematics, Institute for Research in Fundamental Sciences (IPM)
        P. O. Box 19395-5746, Tehran, Iran}\email{sh.mohsenipour@gmail.com}

\address{The Hebrew University of Jerusalem, Einstein Institute of Mathematics,
Edmond J. Safra Campus, Givat Ram, Jerusalem 91904, Israel}

\address{Department of Mathematics, Hill Center - Busch Campus, Rutgers, The State
University of New Jersey, 110 Frelinghuysen Road, Piscataway, NJ 08854-8019,
USA}\email{shelah@math.huji.ac.il}

\thanks{The research of the first author was in part supported by a grant from IPM (No. 94030403). The research of the second author was partially supported by European Research Council grant 338821.  This is paper 1072 in Shelah's list of publications.}

\subjclass[2000]{03E05, 03E35}
\keywords{Combinatorial set theory, Set mappings}
\begin{abstract} In this paper we study set mappings on 4-tuples. We continue a previous work of Komjath and Shelah by getting new finite bounds on the size of free sets in a generic extension. This is obtained by an entirely different forcing construction. Moreover we prove a ZFC result for set mappings on 4-tuples and also as another application of our forcing construction we give a consistency result for set mappings on triples.
\end{abstract}
\maketitle
\bibliographystyle{amsplain}
\section{Introduction}
In this paper we continue the paper of Komjath-Shelah \cite{shelah645}. Our main objects of study here are set mappings which are, for our current purpose, functions of the type $f\colon[\lambda]^{k}\rightarrow[\lambda]^{<\mu}$ for some natural number $k\geq 1$ and cardinals $\lambda$, $\mu$, which satisfy $f(\bar{x})\cap \bar{x}=\emptyset$ for every $\bar{x}\in[\lambda]^{k}$. The motivation in this part of combinatorial set theory is to know how large free sets exist. A subset $H$ of $\lambda$ is called free if $f(\bar{x})\cap H=\emptyset$ for every $\bar{x}\in[H]^{k}$. The case $k=1$ was settled by Hajnal in \cite{hajnalconjecture} where he showed
that if $\mu<\lambda$, then there is a free set of size $\lambda$ (we call it Theorem A for later references). The case $k>1$ came to attention when Kuratowski and Sierpinski proved that for set mappings on $[\lambda]^{k}$ there always exists a free set of size $k+1$ if and only if $\lambda\geq\mu^{+k}$ (see \cite{ehmr}, Section 45). It is interesting to know that assuming GCH, Erd\"{o}s and Hajnal have shown that if $\lambda\geq\mu^{+k}$, then there is a free set of cardinality $\mu^{+}$ \cite{ehmr}. Coming back to ZFC, Hajnal and M\'{a}t\'{e} \cite{hajnalmate} and later Hajnal (\cite{ehmr}, Section 46) managed to improve Kuratowski and Sierpinski's result for the cases $k=2,3$ by showing that if $\lambda\geq\mu^{+k}$, then $f$ has arbitrary large finite sets (we call it Theorem B for the case k=3). The case $k=4$ remained open until Komjath-Shelah in \cite{shelah645} showed that it is consistent with ZFC that there are set mappings with finite bounds on their size of free sets. More precisely, suppose $\mu$ is a regular cardinal, $\lambda=\mu^{+n}$, $n\in\omega$ and GCH holds for every $\mu^{+l}$ ($l<n$), i.e., $2^{\mu^{+l}}=\mu^{+(l+1)}$, then there exist a natural number $t_{n}$ and a cardinal preserving generic extension in which there is a set mapping $f\colon[\lambda]^{4}\rightarrow[\lambda]^{<\mu}$ with no free set of size $t_{n}$. The bound $t_{n}$ is a Ramsey number which is defined inductively to be the least natural number satisfying the Ramsey relation $t_{n}\rightarrow(t_{n-1},7)^{5}$, $t_{0}=5$, hence $n\mapsto t_{n}$ is essentially the tower function, i.e., iterating exponentiation $n$-times. In general the Ramsey relation $a\rightarrow(b,c)^{r}$ means that the following statement is true. Whenever the $r$-element subsets of an $a$-element set are colored with two colors, say 0 and 1, then either there exists a $b$-element with all its $r$-tuples colored with 0 or there exits a $c$-element subset whose all $r$-tuples colored with 1. Subsequently Komjath-Shelah in the same paper \cite{shelah645} proved another independence result to show that Hajnal-M\'{a}t\'{e} and Hajnal results for cases $k=2,3$ are optimal. In precise terms, let $\mu$ be a regular cardinal, $\lambda=\mu^{+n}, n\in\omega$, and GCH holds for every $\mu^{+l}$ ($l<n$), then there is a cardinal preserving generic extension in which there exits a set mapping $f\colon[\lambda]^{2}\rightarrow[\lambda]^{<\mu}$ with no free set of size $\omega$. This easily implies the similar result for $k=3$. Coming back to the case $k=4$ we should say that this is not the end of the story. Now we definitely have the task of improving the bound $t_{n}$ in front of us which would be more serious if we notice that the following question is open:

\begin{question}[Gillibert-Wehrung \cite{gillibert-wehrung}]\label{gilibert-wehrung} Suppose $\mu$ is an infinite cardinal and $\lambda=\mu^{+4}$. Does any set mapping $f\colon[\lambda]^{4}\rightarrow[\lambda]^{<\mu}$ have a free set of size 7?
\end{question}
Recall that the Kuratowski-Sierpinski Theorem guarantees the existence of a free set of size 5. In fact Hajnal and M\'{a}t\'{e} \cite{hajnalmate} had asked in 1975 whether the above set mapping $f\colon[\lambda]^{4}\rightarrow[\lambda]^{<\mu}$ has a free set of size 6. This question had remained open until 2008 when Gillibert by using algebraic tools of completely different nature, gave a positive answer to it \cite{gillibert-thesis}. We refer the reader to \cite{gillibert-wehrung} for a self contained proof of this fact and other interesting results.

Now we are ready to briefly describe what we have done in this paper. In Section 2 we obtain another bound $s_n$ instead of $t_n$ with a different forcing construction. The new bound $s_n$ is also a Ramsey number and is defined to be the least natural number satisfying $s_n\rightarrow(5)^{3}_{3^{n+1}}$ and essentially is triple exponentiation. In general for natural numbers $a,b,c$ and $r$, the relation $a\rightarrow(b)^{r}_{c}$ means that the following assertion is true. Whenever the $r$-element subsets of an $a$-element set are colored with $c$ colors, then there is a $b$-element subset whose $r$-element subsets have the same color. Note that by the Erd\"{o}s and Hajnal theorem quoted above, GCH must fail in the generic extension. We manage to control our construction in such a way that in the generic extension we will have $2^{\mu}=2^{\mu^{+}}$ $=\dots=2^{\mu^{+(n-1)}}=\lambda$. In Section 3 we will compare the two bounds $s_n$ and $t_n$ asymptotically by using Ramsey theory to show that $s_n$ is much better than $t_n$ when $n$ tends to infinity. Motivated by our main results in Section 2, we consider set mappings mappings on 4-tuples with a restriction on the location of the image in Section 4 and get a ZFC result. In fact in both of the forcing constructions in this paper and in Komjath-Shelah's paper \cite{shelah645}, we find that the constructed set mappings $f\colon[\lambda]^{4}\rightarrow[\lambda]^{<\mu}$ in the generic extension have this property: for all $x_0<x_1<x_2<x_3$ in $[\lambda]^{4}$, $f\{x_0,x_1,x_2,x_3\}\subset(x_1,x_2)$. By modifying Hajnal's proof of Theorem B we show that this is necessary, more precisely we show that if $\lambda\geq\mu^{+3}$, then any set mapping $f\colon[\lambda]^{4}\rightarrow[\lambda]^{<\mu}$ with the additional property, for all $x_0<x_1<x_2<x_3$ in $[\lambda]^{4}$, $f\{x_0,x_1,x_2,x_3\}\cap(x_1,x_2)=\emptyset$, will have arbitrary large finite free sets. In Section 5 we consider a similar situation for set mapping on triples. This time we are motivated by Komjath-Shelah's second construction in \cite{shelah645} and then we will deal with set mappings on triples with a restriction on the location of the image. As an application of our forcing in Section 2, we obtain a negative consistency result.\\

\noindent$\mathbf{Notation.}$ Cardinals are identified with initial ordinals. If $S$ is a set and $\kappa$ is a cardinal, then $[S]^{\kappa}=\{X\subset S:|X|=\kappa\}$, $[S]^{<\kappa}=\{X\subset S:|X|<\kappa\}$, $[S]^{\leq\kappa}=\{X\subset S:|X|\leq\kappa\}$. If $A$ and $B$ are subsets of an ordered set, then $A<B$ means that $x<y$ whenever $x\in A$ and $y\in B$. Let $F\colon[A]^{k}\rightarrow B$ be a function on finite subsets, we always write $F\{a_{1},\dots,a_{k}\}$ rather than $F(\{a_{1},\dots,a_{k}\})$. If $\lambda$ is a cardinal and $\alpha\in\lambda$, then by $(\alpha,+\infty)$ we mean $\{x\in\lambda:x>\alpha\}$.

\section{The Main Theorems}\label{mainsection}

We begin with a definition and two simple lemmas.
\begin{definition} \label{def1}Suppose $A$ is a set, $k\geq 2$, $F\colon[A]^{k}\rightarrow\mathcal{P}\left({A}\right)$ is a set mapping and $\Gamma$ is a set of functions of the form $\rho\colon[A]^{2}\rightarrow L_{\rho}$ where $L_{\rho}$ is a linear ordering with the order $<_{\rho}$; pedantically $\Gamma$ is the set of pairs $(\rho,L_{\rho})$ for $\rho\in\Gamma$. We say that $\Gamma$ $k$-generates $F$ if far all $\bar{\gamma}\in[A]^{k}$ we have $x\in F(\bar{\gamma})$ if and only if $x\in A$ and
\begin{center}
$(\forall\rho\in\Gamma)(\forall\gamma\in\bar{\gamma})(\exists\gamma^{'},\gamma^{''}\in\bar{\gamma})\,\bigr{[}\,\gamma^{'}\neq\gamma^{''}\,\wedge\,\,
\rho\{x,\gamma\}\leq_{\rho}\rho\{\gamma^{'},\gamma^{''}\}\bigr{]}$.
\end{center}
\end{definition}
We note that in this paper all the linear orderings $L_{\rho}$ will be some cardinals with their natural $\in$-orderings. It is evident from the above definition that for any $\Gamma$ as above and any $k\geq 2$ there is a unique set mapping $F\colon[A]^{k}\rightarrow\mathcal{P}\left({A}\right)$ such that $\Gamma$ $k$-generates $F$, so we may denote it by $F=F_{\,\Gamma}$.
\begin{definition}
In \ref{def1}, suppose $B\subset A$, by $\Gamma|_{B}$ we mean $\{\rho|_{B}:\rho\in\Gamma\}$.
\end{definition}
The following two very easy lemmas will be very useful in the sequel.
\begin{lemma}\label{lemma} Let $F_{1}\colon[A_{1}]^{k}\rightarrow\mathcal{P}\left({A_{1}}\right)$, $F_{2}\colon[A_{2}]^{k}\rightarrow\mathcal{P}\left({A_{2}}\right)(k\geq2)$ be set mappings with $A_{1}\subset A_{2}$. Let $\Gamma_{1},\Gamma_{2}$ be sets of functions such that $\Gamma_{i}$ $k$-generates $F_{i}$ $(i=1,2)$ and $\Gamma_{2}|_{A_{1}}\subset\Gamma_{1}$, then $\forall\bar{\gamma}\in[A_{1}]^{k}$ we have $F_{1}(\bar{\gamma})\subset F_{2}(\bar{\gamma})$.
\end{lemma}
\begin{lemma}\label{lemma2}
Let $F_{1}\colon[A_{1}]^{k}\rightarrow\mathcal{P}\left({A_{1}}\right)$, $F_{2}\colon[A_{2}]^{k}\rightarrow\mathcal{P}\left({A_{2}}\right)(k\geq2)$ be set mappings with $A_{1}\subset A_{2}$. Let $\Gamma_{1},\Gamma_{2}$ be sets of functions such that $\Gamma_{i}$ $k$-generates $F_{i}$ $(i=1,2)$, $\Gamma_{2}|_{A_{1}}=\Gamma_{1}$, then $\forall\bar{\gamma}\in[A_{1}]^{k}$ we have $F_{1}(\bar{\gamma})=F_{2}(\bar{\gamma})\cap A_{1}$.
\end{lemma}

\begin{theorem}\label{main}
Assume that $n<\omega$, $k\geq 2$ and $\lambda=\mu^{+n}$ for some regular cardinal $\mu=\mu^{<\mu}$. Suppose
\begin{center}
$2^{\mu}=\mu^{+}, 2^{\mu^{+}}=\mu^{++},\dots,2^{\mu^{+(n-1)}}=\mu^{+n}$ (when $n>0$).
\end{center}
Then there is a notion of forcing $\mathbb{P}_{n}=\mathbb{P}_{n}^{\mu,\lambda}$ such that $|\mathbb{P}_{n}|=\lambda$, $\mathbb{P}_{n}$ is $(<\mu)$-complete and collapses no cardinal and in $V^{\mathbb{P}_{n}}$:
\begin{itemize}
\item[(i)]there are functions
\begin{center}
$\rho_{0}\colon[\lambda]^{2}\rightarrow\mu^{+n}, \rho_{1}\colon[\lambda]^{2}\rightarrow\mu^{+(n-1)},\dots,\rho_{n}\colon[\lambda]^{2}\rightarrow\mu$
\end{center}
such that if $f\colon[\lambda]^{k}\rightarrow\mathcal{P}(\lambda)$ is $k$-generated by $\Gamma=\{\rho_{0},\dots,\rho_{n}\}$, then $\forall\bar{\gamma}\in[\lambda]^{k}$ we have $|f(\bar{\gamma})|<\mu$. This means that $f\colon[\lambda]^{k}\rightarrow[\lambda]^{<\mu}$.
\item[(ii)]$2^{\mu}=2^{\mu^{+}}=\dots=2^{\mu^{+(n-1)}}=\lambda.$
\end{itemize}
\end{theorem}
\begin{proof}
By induction on $n$. We first prove the case $n=0$ in ZFC. Set $\rho_{0}\colon[\lambda]^{2}\rightarrow\lambda$ by $\rho_{0}\{x,y\}=\max\{x,y\}$ for all $x,y\in\lambda$. Let $f\colon[\lambda]^{k}\rightarrow\mathcal{P}(\lambda)$ is $k$-generated by $\{\rho_{0}\}$. It is very easy to check that $\forall\bar{\gamma}\in[\lambda]^{k}$ we have $|f(\bar{\gamma})|<\lambda$.

We are also able to prove the case $n=1$ in ZFC. Set $\rho_{0}$ as in the previous case, $\rho_{0}\{x,y\}=\max\{x,y\}$. It is well known that there is a function $\rho_{1}\colon [\lambda]^{2}\rightarrow\mu$ such that if $\alpha<\beta<\mu^{+}=\lambda$ and $\nu<\mu$, then $|\bigr\{\xi\leq\alpha:\rho_{1}\{\xi,\alpha\}\leq\nu\bigr\}|<\mu$\,\,($\clubsuit$). Now let $\bar{\gamma}\in[\lambda]^{k}$. Set $v^{*}=\max\bigr\{\rho_{1}\{\gamma,\gamma^{'}\}:\gamma,\gamma^{'}\in\bar{\gamma}\bigr\}$, $\gamma^{*}=\max\bar{\gamma}$. Since $f$ is $k$-generated by $\{\rho_{0},\rho_{1}\}$, we have $f(\bar{\gamma})\subset\bigr\{x\leq\gamma^{*}:\rho_{1}\{x,\gamma^{*}\}\leq\nu^{*}\bigr\}$. Now ($\clubsuit$) implies that $|\bigr\{x\leq\gamma^{*}:\rho_{1}\{x,\gamma^{*}\}\leq\nu^{*}\bigr\}|<\mu$, thus $|f(\bar{\gamma})|<\mu$.

Now assume that the theorem is true for $n\geq1$. We prove it for $n+1$. Recall that $\mu$ is a regular cardinal, $\mu=\mu^{<\mu}$, $2^{\mu}=\mu^{+},\dots,2^{\mu^{+n}}=\mu^{+(n+1)}=\lambda$. Obviously $\mu^{+}={(\mu^{+})}^{<(\mu^{+})}$, so by the induction hypothesis there is a notion of forcing $\mathbb{P}_{n}=\mathbb{P}_{n}^{\mu^{+},\lambda}$ such that $\mathbb{P}_{n}$ is $(<\mu^{+})$-complete, of cardinality $\lambda$ and collapses no cardinal and in $W=V^{\mathbb{P}_{n}}$:
\begin{itemize}
\item[(i)]there are functions
\begin{center}
$\rho_{0}\colon[\lambda]^{2}\rightarrow\mu^{+(n+1)}, \rho_{1}\colon[\lambda]^{2}\rightarrow\mu^{+n},\dots,\rho_{n}\colon[\lambda]^{2}\rightarrow\mu^{+}$
\end{center}
such that if $f\colon[\lambda]^{k}\rightarrow\mathcal{P}(\lambda)$ is $k$-generated by $\{\rho_{0},\dots,\rho_{n}\}$, then $\forall\bar{\gamma}\in[\lambda]^{k}$ we have $|f(\bar{\gamma})|<\mu^{+}$.
\item[(ii)]$2^{\mu^{+}}=\dots=2^{\mu^{+n}}=\lambda.$
\end{itemize}
We define a forcing notion $\mathbb{P}$ in $W=V^{\mathbb{P}_{n}}$ which is $(<\mu)$-complete and has the $\mu^{+}$-c.c. such that for a generic $G\subset\mathbb{P}$ over $W$ we can find $\rho_{n+1}\colon[\lambda]^{2}\rightarrow\mu$ in the generic extension $W[G]$ so that if  $f\colon[\lambda]^{k}\rightarrow\mathcal{P}(\lambda)$ is $k$-generated by $\{\rho_{0},\dots,\rho_{n},\rho_{n+1}\}$, then $\forall\bar{\gamma}\in[\lambda]^{k}$ we have $|f(\bar{\gamma})|<\mu$. Let $(\mathbb{P},\leq)$ consists of the triples of the form $\langle s,g,\varrho\rangle$ when $s\in[\lambda]^{<\mu}$, $\varrho\colon[s]^{2}\rightarrow\mu$, $g\colon[s]^{k}\rightarrow\mathcal{P}(s)$ and $\{\rho_{0}|_{s},\rho_{1}|_{s},\dots,\rho_{n}|_{s},\varrho\}$ $k$-generates $g$. For facility in notation we sometimes denote the condition $p=\langle s,g,\varrho\rangle$ by $p=\langle s_{p},g_{p},\varrho_{p}\rangle$.

For $p=\langle s,g,\varrho\rangle$ and $p^{'}=\langle s^{'},g^{'},\varrho^{'}\rangle$ in $(\mathbb{P},\leq)$ we say that $p^{'}$ extends $p$ $(p^{'}\leq p)$ if and only if $s^{'}\supset s$, $g^{'}|_{[s]^{k}}=g$ and $\varrho^{'}|_{[s]^{2}}=\varrho$.\\

\emph{Claim 1.} $(\mathbb{P},\leq)$ is ($<\mu$)-complete.\\

\emph{Proof of Claim 1.} Let $\eta$ be a cardinal $<\mu$ and $\bigr{\langle} p_{i}=\langle s_{i},g_{i},\varrho_{i}\rangle;i<\eta\bigr{\rangle}$ be a decreasing sequence of conditions in $\mathbb{P}$. Set
\[
s=\bigcup_{i<\eta}s_{i},\,\,\,g=\bigcup_{i<\eta}g_{i},\,\,\,\varrho=\bigcup_{i<\eta}\varrho_{i}.
\]
We show that $p=\langle s,g,\varrho\rangle$ is a condition in $\mathbb{P}$ extending all $p_{i},i<\eta$. By regularity of $\mu$ we have $|s|<\mu$. Now suppose $g^{'}:[s]^{k}\rightarrow\mathcal{P}(s)$ is the function $k$-generated by $\{\rho_{0}|_{s},\dots,\rho_{n}|_{s},\varrho\}$. We prove $g=g^{'}$. Let $\bar{\gamma}\in[s]^{k}$ and also let $i_{0}<\eta$ be such that $\bar{\gamma}\in[s_{i_{0}}]^{k}$. Lemma \ref{lemma2} says that for all $i_{0}\leq i<\eta$, $g^{'}(\bar{\gamma})\cap s_{i}=g_{i}(\bar{\gamma})$. Thus
\begin{eqnarray*}
g^{'}(\bar{\gamma})=g^{'}(\bar{\gamma})\cap s=g^{'}(\bar{\gamma})\cap(\bigcup_{i_{0}\leq i<\eta}s_{i})&=&\bigcup_{i_{0}\leq i<\eta}(g^{'}(\bar{\gamma})\cap s_{i})\\                                     &=&\bigcup_{i_{0}\leq i<\eta}g_{i}(\bar{\gamma})=g(\bar{\gamma}).
\end{eqnarray*}
Also by definition of $p$, it is clear that $p$ extends all $p_{i},i<\eta$. This proves Claim 1.\\

\emph{Claim 2.} $(\mathbb{P},\leq)$ has $\mu^{+}$-c.c.\\

\emph{Proof of Claim 2.} Assume that $p_{i}=\langle s_{i},g_{i},\varrho_{i}\rangle \in (\mathbb{P},\leq)$ for $i<\mu^{+}$. Using a standard $\Delta$-system argument we can suppose that $s_{i}=a\cup b_{i}$ for the pairwise disjoint sets $a, b_{i},i<\mu^{+}$. Let $f\colon[\lambda]^{k}\rightarrow\mathcal{P}(\lambda)$ be $k$-generated by $\{\rho_{0},\dots,\rho_{n}\}$. By Lemma \ref{lemma} for any $\bar{\gamma}\in[a]^{k}$, $g_{i}(\bar{\gamma})$ is a subset of $f(\bar{\gamma})$ and by definition $g_{i}(\bar{\gamma})$ is of cardinality $<\mu$. As $|f(\bar{\gamma})|\leq\mu$ (by the induction hypothesis) and $|a|<\mu$ we can assume, by $\mu^{<\mu}=\mu$, that $g_{i}|_{[a]^{k}}$ is the same for $i<\mu^{+}$. Also considering the fact that range$(\varrho_{i})\subset\mu$ for $i<\mu^{+}$, we can assume that $\varrho_{i}|_{[a]^{2}}$ is the same for $i<\mu^{+}$.

Now we will show that for any $i,j<\mu^{+}$, $p_{i}$ and $p_{j}$ are compatible. Set $s=s_{i}\cup s_{j}=a\cup b_{i}\cup b_{j}$. We intend to find $q=\langle s, g, \varrho\rangle\in\mathbb{P}$ such that $q$ extends both $p_{i},p_{j}$. Define $\varrho\colon[s]^{2}\rightarrow\mu$ by $\varrho\supset\varrho_{i},\varrho_{j}$ and for $\{\alpha_{1},\alpha_{2}\}\in[s]^{2}-[s_{i}]^{2}-[s_{j}]^{2}$,
\begin{equation}\label{rho}
\varrho\{\alpha_{1},\alpha_{2}\}=\sup\bigr{(}\varrho_{i}^{''}[s_{i}]^{2}\cup\varrho_{j}^{''}[s_{j}]^{2}\bigr{)}+1.
\end{equation}
Note that since $\mu$ is regular and $|s_{i}|,|s_{j}|<\mu$, the sets $\varrho_{i}^{''}[s_{i}]^{2}$, $\varrho_{j}^{''}[s_{j}]^{2}$ are bounded in $\mu$, so the above definition is well defined. Now let $g\colon[s]^{k}\rightarrow\mathcal{P}(s)$ be the function $k$-generated by $\{\rho_{0}|_{s},\dots,\rho_{n}|_{s},\varrho\}$. This completes the definition of $q=\langle s, g, \varrho\rangle$. Clearly $q\in\mathbb{P}$. We must show that $q$ extends $p_{i},p_{j}$. By symmetry it is enough to show $q$ extends $p_{i}$. As $s_{i}=a\cup b_{i}\subset s$ and $\varrho_{i}\subset\varrho$, this would be done if we could show $\forall\bar{\gamma}\in[s_{i}]^{k}\,\, g_{i}(\bar{\gamma})=g(\bar{\gamma})$. Lemma \ref{lemma} tells us that $g_{i}(\bar{\gamma})\subset g(\bar{\gamma})$. By definition,  $g(\bar{\gamma})\subset s$, $g_{i}(\bar{\gamma})\subset s_{i}$ and from Lemma \ref{lemma2} it follows that $g_{i}(\bar{\gamma})=g(\bar{\gamma})\cap s_{i}$. So it remains to show that $g(\bar{\gamma})\cap (s_{j}\setminus s_{i})=\emptyset$. By way of contradiction assume there is $x\in g(\bar{\gamma})\cap b_{j}$. Note that if $\bar{\gamma}\in[a]^{k}$, then $x\in g_{j}(\bar{\gamma})=g_{i}(\bar{\gamma})\subset s_{i}$ which implies that $x\in b_{j}\cap s_{i}$. This contradicts the disjointness of $s_{i},b_{j}$. So suppose there exists $\gamma \in \bar{\gamma}$ such that $\gamma\notin a$. This means that
\[
\{x,\gamma\}\in[s]^{2}-[s_{i}]^{2}-[s_{j}]^{2},
\]
then by (\ref{rho}) we have $\varrho\{x,\gamma\}>\sup\varrho_{i}^{''}[s_{i}]^{2}$. In other words for all $\{\gamma^{'},\gamma^{''}\}\in[\bar{\gamma}]^{2}$ we have $\varrho\{x,\gamma\}>\varrho\{\gamma^{'},\gamma^{''}\}$ which implies that $x\notin g(\bar{\gamma})$, a contradiction. Thus we have proved that $\forall\bar{\gamma}\in[s_{i}]^{k}\,\, g_{i}(\bar{\gamma})=g(\bar{\gamma})$ and this shows that $q$ extends $p_{i}$ which completes the proof of Claim 2.\\

Now let $G$ be a $\mathbb{P}$-generic filter over $W$. \\

\emph{Claim 3.} $\displaystyle{\bigcup_{p\in G}}s_{p}=\lambda$.\\

\emph{Proof of Claim 3.} Obviously $\bigcup_{p\in G}s_{p}\subset\lambda$. For any $\alpha\in \lambda$, let $D_{\alpha}=\{p\in\mathbb{P}:\alpha\in s_{p}\}$. We will show that $D_{\alpha}$ is dense in $\mathbb{P}$. Assume that $q=\langle s,g,\varrho\rangle\in\mathbb{P}$ with $\alpha\notin s$. Let $s^{'}=s\cup\{\alpha\}$. We define $\varrho^{'}\colon[s_{1}]^{2}\rightarrow\mu$ as follows. $\varrho^{'}|_{s}=\varrho$ and for $\{\alpha,\gamma\}\in[s^{'}]^{2}-[s]^{2}$, set $\varrho^{'}\{\alpha,\gamma\}=\sup(\varrho^{''}[s]^{2})+1$. By regularity of $\mu$ and $|s|<\mu$ this is well defined. Now let $g^{'}\colon[s^{'}]^{2}\rightarrow\mathcal{P}(s^{'})$ be the function $k$-generated by $\{\rho_{0}|_{s^{'}},\dots,\rho_{n}|_{s^{'}},\varrho^{'}\}$. So $p=\langle s^{'},g^{'},\varrho^{'}\rangle\in\mathbb{P}\cap D_{\alpha}$. We show that $p\leq q$. Obviously $s\subset s^{'}$, $\varrho\subset\varrho^{'}$. Suppose $\bar{\gamma}\in[s]^{k}$, we must prove that $g(\bar{\gamma})=g^{'}(\bar{\gamma})$. By Lemmas \ref{lemma} and \ref{lemma2} we have $g(\bar{\gamma})\subset g^{'}(\bar{\gamma})$, $g(\bar{\gamma})=g^{'}(\bar{\gamma})\cap s$. Thus it remains to show that $\alpha\notin g^{'}(\bar{\gamma})$. This is so because for each $\gamma\in\bar{\gamma}$ and each $\{\gamma^{'},\gamma^{''}\}\in[\bar{\gamma}]^{2}$, the definition of $\varrho^{'}$ implies that $\varrho^{'}\{\alpha,\gamma\}>\varrho^{'}\{\gamma^{'},\gamma^{''}\}$. This proves the density of $D_{\alpha}$ from which it follows that $\exists r\in D_{\alpha}\cap G$. So $\alpha\in s_{r}\subset\bigcup_{p\in G}s_{p}$. Therefore $\bigcup_{p\in G}s_{p}=\lambda$ and Claim 3 is proved.\\

Now we are ready to introduce $\rho_{n+1}\colon[\lambda]^{2}\rightarrow\mu$ and $f\colon[\lambda]^{k}\rightarrow\mathcal{P}(\lambda)$. Set
\[
\rho_{n+1}=\bigcup_{p\in G}\varrho_{p}, \,\,\,f=\bigcup_{p\in G}g_{p}.
\]

\noindent We show that\\

\noindent (1) $\forall\bar{\gamma}\in[\lambda]^{k}|f(\bar{\gamma})|<\mu$,

\noindent (2) $f$ is $k$-generated by $\{\rho_{0},\dots,\rho_{n+1}\}$.\\

\noindent Let $\bar{\gamma}\in[\lambda]^{k}$ and $p\in G$ such that $\bar{\gamma}\in[s_{p}]^{k}$. Then $f(\bar{\gamma})=g_{p}(\bar{\gamma})\subset s_{p}\in[\lambda]^{<\mu}$. This proves (1). Now let $f^{'}\colon[\lambda]^{k}\rightarrow\mathcal{P}(\lambda)$ be the function $k$-generated by $\{\rho_{0},\dots,\rho_{n+1}\}$. Let $\bar{\gamma}\in[\lambda]^{k}$ and $p\in G$ such that $\bar{\gamma}\in[s_{p}]^{k}$. By Lemma \ref{lemma2}, we have $f^{'}(\bar{\gamma})\cap s_{p}=g_{p}(\bar{\gamma})$. But since $G$ is a filter, for every such $p,q\in G$, $g_{p}(\bar{\gamma})=g_{q}(\bar{\gamma})$. Putting together this with $\bigcup_{p\in G}s_{p}=\lambda$ we can deduce (2).  In other words
\begin{eqnarray*}
f^{'}(\bar{\gamma})=f^{'}(\bar{\gamma})\cap\lambda=f^{'}(\bar{\gamma})\cap\bigcup_{p\in G}s_{p}&=&\bigcup_{p\in G}\bigr{(}f^{'}(\bar{\gamma})\cap s_{p}\bigr{)}\\
                                                                                               &=&\bigcup_{p\in
                                                                                               G}g_{p}(\bar{\gamma})=f(\bar{\gamma}).\end{eqnarray*}
Now let $\dot{\mathbb{P}}$ be a $\mathbb{P}_{n}$-name for $\mathbb{P}$. Then $\mathbb{P}_{n}\ast{\dot{\mathbb{P}}}$ is $(<\mu)$-complete and collapses no cardinal and has cardinality $\lambda$ (recall the definition of $\mathbb{P}$ in $V^{\mathbb{P}_{n}}$ and notice that $V^{\mathbb{P}_{n}}\models \lambda^{<\lambda}=\lambda$) and by what we have done so far $V^{\mathbb{P}_{n}\ast{\dot{\mathbb{P}}}}$ satisfies the requirement (i) of the theorem for $n+1$. Now observe that for $l=1,\dots,n$ we have $V^{\mathbb{P}_{n}}\models({\lambda})^{\mu^{+l}}=(2^{\mu^{+}})^{\mu^{+l}}={2}^{\mu^{+}}=\lambda$. Since $\mathbb{P}$ has the $\mu^{+}$-c.c. in $V^{\mathbb{P}_{n}}$, it is well known that we have $V^{\mathbb{P}_{n}\ast{\dot{\mathbb{P}}}}\models({2})^{\mu^{+l}}=\lambda$\,\,$(\heartsuit_{1})$. This implies that for $k=0,\dots,n$, $V^{\mathbb{P}_{n}\ast{\dot{\mathbb{P}}}}\models(\lambda)^{\mu^{+k}}=(2^{\mu^{+}})^{\mu^{+k}}=\lambda$\,\,$(\heartsuit_{2})$. Also by $(<\mu)$-completeness of $\mathbb{P}$ in $V^{\mathbb{P}_{n}}$ we have $V^{\mathbb{P}_{n}\ast{\dot{\mathbb{P}}}}\models2^{<\mu}=\mu$\,\,$(\heartsuit_{3})$. Assume that $\mathbb{Q}$ is the poset of all functions $p\colon\lambda\times\mu\rightarrow2$ in $V^{\mathbb{P}_{n}\ast\dot{\mathbb{P}}}$ such that $|p|<\mu$. Let $\dot{\mathbb{Q}}$ be a $(\mathbb{P}_{n}\ast\dot{\mathbb{P}})$-name for $\mathbb{Q}$. It is well known that from $\heartsuit_{2}$ for $k=0$ together with $\heartsuit_{3}$ we can deduce $V^{\mathbb{P}_{n}\ast\dot{\mathbb{P}}\ast\dot{\mathbb{Q}}}\models2^{\mu}=\lambda$. Since $\mathbb{Q}$ has the $\mu^{+}$-c.c. in $V^{\mathbb{P}_{n}\ast\dot{\mathbb{P}}}$, from $\heartsuit_{2}$ it follows that $V^{\mathbb{P}_{n}\ast{\dot{\mathbb{P}}}\ast{\dot{\mathbb{Q}}}}\models({2})^{\mu^{+l}}=\lambda$ for $l=1,\dots,n$ (as in the case of $\heartsuit_{1}$). Now we define $\mathbb{P}_{n+1}=\mathbb{P}_{n+1}^{\mu,\lambda}={\mathbb{P}_{n}}^{\!\!\!\mu^{+}\!\!,\lambda}\ast{\dot{\mathbb{P}}}\ast{\dot{\mathbb{Q}}}$. Surely $\mathbb{P}_{n+1}$ is $(<\mu)$-complete and collapses no cardinal and has cardinality $\lambda$ and satisfies both the requirements (i), (ii) of the theorem for $n+1$. So we finished the proof.
\end{proof}

It is interesting to note that the referee has suggested an alternate exposition of the proof of Theorem \ref{main}. In fact the proof of Theorem \ref{main} can be split into two steps, first proving that if ($\mu=\mu^{<\mu}$), $2\leq k<\omega$, $F\colon[\lambda]^{k}\rightarrow[\lambda]{\leq\mu}$, then a $(<\mu)$-complete cardinal-preserving forcing adds a function $h^{*}\colon [\lambda]^{2}\rightarrow\mu$ such that if for $x\in[\lambda]^{k}$
\[
F^{*}(x)=\{\xi\in F(x)\colon h^{*}(\xi,\gamma)\leq h^{*}(x)(\gamma\in x)\}
\]
then $|F^{*}(x)|<\mu$. Here $h^{*}(x)=\max\{h^{*}(y)\colon y\in[x]^{2}\}$. We also define a forcing notion $(Q,\leq)$ as follows. The elements are tuples $(s,h)$ such that $s\in[\lambda]^{<\mu}$, $h\colon[s]^{2}\rightarrow\mu$. Also $(s^{'},h^{'})\leq(s,h)$ if $s^{'}\supset s$, $h=h^{'}|[s]^{2}$ as well as there are no $x\in[s]^{k},\,\, \xi\in(F(x)\cap s^{'})-s$ such that for each $\gamma\in x,\,\, h^{'}(\xi , \gamma)\leq h(x)$. Then we argue as in the proof of Theorem \ref{main}. Finally Theorem \ref{main} can be obtained by iteration. We leave it to the reader to reproduce the proof via this approach.\\

The following definition is useful in presenting the proof of the next theorem.

\begin{definition}
Assume that $A$ is a set with $|A|>3$, $\langle L,<\rangle$ is a linear order and $\rho\colon[A]^{2}\rightarrow L$. Let $B\subset A$, $x\in A\setminus B$. We say that $x$ is $\rho$-close to $B$, if
\[
(\forall\gamma\in B)(\exists \gamma^{'},\gamma^{''}\in B) \bigr{[}(\gamma^{'}\neq\gamma^{''})\,\,\wedge\,\,\rho\{x,\gamma\}\leq\rho\{\gamma^{'},\gamma^{''}\}\bigr{]}.
\]
\end{definition}
\begin{theorem}
Assume that $n<\omega, k=4$ and $\lambda=\mu^{+n}$ for some regular cardinal $\mu$ and $s_{n}$ is the the least number satisfying the Ramsey relation $s_{n}\rightarrow(5)^{3}_{3^{n+1}}$. Let $\mathbb{P}_{n}=\mathbb{P}^{\mu,\lambda}_{n}$ be as in Theorem \ref{main}. Then in $V^{\mathbb{P}_{n}}$ there is a set mapping $F\colon[\lambda]^{4}\rightarrow[\lambda]^{<\mu}$ such that $F$ has no free set of cardinality $s_{n}$.
\end{theorem}
\begin{proof}By Theorem \ref{main}, in $V^{\mathbb{P}_{n}}$ there are functions $F\colon[\lambda]^{4}\rightarrow[\lambda]^{<\mu}$ and
\[
\rho_{0}\colon[\lambda]^{2}\rightarrow\mu^{+n},\,\dots\,,\,\rho_{n}\colon[\lambda]^{2}\rightarrow\mu
\]
such that $F$ is $4$-generated by $\Gamma=\{\rho_{0},\dots,\rho_{n}\}$. We shall show that $F$ has no free set of cardinality $s_{n}$. By way of contradiction suppose not and let $A=\{\gamma_{1},\dots,\gamma_{s_{n}}\}$ be a free set for $F$ with $s_{n}$ elements. We are going to define a partition relation $\sim$ on $[A]^{3}$. Let $\rho\in\Gamma$, $\{\alpha,\beta,\gamma\}\in[A]^{3}$ with $\alpha<\beta<\gamma$. At least one of the following three possibilities will occur:\\

$\mathrm{(1)}\,\,\,\max\bigr{\{}\rho\{\alpha,\beta\},\rho\{\beta,\gamma\}\bigr{\}}\leq\rho\{\alpha,\gamma\}$,\\

$\mathrm{(2)}\,\,\,\max\bigr{\{}\rho\{\alpha,\beta\},\rho\{\alpha,\gamma\}\bigr{\}}\leq\rho\{\beta,\gamma\}$,\\

$\mathrm{(3)}\,\,\,\max\bigr{\{}\rho\{\beta,\gamma\},\rho\{\alpha,\gamma\}\bigr{\}}\leq\rho\{\alpha,\beta\}$.\\

\noindent We say that $\{\alpha,\beta,\gamma\}$ has $\rho$-\textit{type} $l$ $(l=1,2,3)$, if the possibility $(l)$ occurs, and $l$ is minimal. Now for $\bar{\gamma}_{1},\bar{\gamma}_{2}\in[A]^{3}$, we put $\bar{\gamma}_{1}\sim\bar{\gamma}_{2}$ iff $\forall \rho\in\Gamma$, $\bar{\gamma}_{1},\bar{\gamma}_{2}$ have the same $\rho$-type. The number of equivalence classes is at most $3^{n+1}$, so from $|A|=s_{n}, s_{n}\rightarrow(5)^{3}_{3^{n+1}}$ it follows that there is a 5-element homogenous set $B=\{\alpha_{1},\alpha_{2},\alpha_{3},\alpha_{4},\alpha_{5}\}$. Assume that $\alpha_{1}<\alpha_{2}<\alpha_{3}<\alpha_{4}<\alpha_{5}$.\\

\textit{Claim.} For every $\rho\in\Gamma$, $\alpha_{3}$ is $\rho$-close to $\{\alpha_{1},\alpha_{2},\alpha_{4},\alpha_{5}\}$. \\

\textit{Proof of Claim.} Let $\rho\in\Gamma$, there are three cases to deal with:\\

\underline{Case 1}: For all $\bar{\gamma}\in[B]^{3}$, $\bar{\gamma}$ has $\rho$-type 1. In this case we have the following relations:
\begin{center}
$\rho\{\alpha_{3},\alpha_{1}\}\leq\rho\{\alpha_{1},\alpha_{4}\},\,\,\,\rho\{\alpha_{3},\alpha_{2}\}\leq\rho\{\alpha_{2},\alpha_{4}\},$
\end{center}
\begin{center}
$\rho\{\alpha_{3},\alpha_{4}\}\leq\rho\{\alpha_{4},\alpha_{1}\},\,\,\,\rho\{\alpha_{3},\alpha_{5}\}\leq\rho\{\alpha_{5},\alpha_{1}\}.$ \\
\end{center}

\underline{Case 2}: For all $\bar{\gamma}\in[B]^{3}$, $\bar{\gamma}$ has $\rho$-type 2. In this case we have the following relations:
\begin{center}
$\rho\{\alpha_{3},\alpha_{1}\}\leq\rho\{\alpha_{4},\alpha_{5}\},\,\,\,\rho\{\alpha_{3},\alpha_{2}\}\leq\rho\{\alpha_{4},\alpha_{5}\},$
\end{center}
\begin{center}
$\rho\{\alpha_{3},\alpha_{4}\}\leq\rho\{\alpha_{4},\alpha_{5}\},\,\,\,\rho\{\alpha_{3},\alpha_{5}\}\leq\rho\{\alpha_{4},\alpha_{5}\}.$ \\
\end{center}

\underline{Case 3}: For all $\bar{\gamma}\in[B]^{3}$, $\bar{\gamma}$ has $\rho$-type 3. In this case we have the following relations:
\begin{center}
$\rho\{\alpha_{3},\alpha_{1}\}\leq\rho\{\alpha_{1},\alpha_{2}\},\,\,\,\rho\{\alpha_{3},\alpha_{2}\}\leq\rho\{\alpha_{1},\alpha_{2}\},$
\end{center}
\begin{center}
$\rho\{\alpha_{3},\alpha_{4}\}\leq\rho\{\alpha_{1},\alpha_{2}\},\,\,\,\rho\{\alpha_{3},\alpha_{5}\}\leq\rho\{\alpha_{1},\alpha_{2}\}.$ \\
\end{center}

Therefore in all cases we have shown that $\alpha_{3}$ is $\rho$-close to $\{\alpha_{1},\alpha_{2},\alpha_{4},\alpha_{5}\}$. This proves Claim.

Now from Claim it follows that $\alpha_{3}\in F\{\alpha_{1},\alpha_{2},\alpha_{4},\alpha_{5}\}$ which violates the freeness of $A$, a contradiction.
\end{proof}

\section{Some Ramsey Considerations}

In this section we show that when $n$ tends to infinity, $s_n$ give us a better bound than $t_n$. First we fix our notation for representing Ramsey numbers:

\begin{eqnarray*}
&& R_{k}(l_{1},\dots,l_{r})=\min\bigr\{n_{0}:\mathrm{for}\,\,\, n\geq n_{0}\,\,\, n\rightarrow(l_{1},\dots,l_{r})^{k}\bigr\},\\
&& R_{k}(l;r)=\min\bigr\{n_{0}:\mathrm{for}\,\,\, n\geq n_{0}\,\,\, n\rightarrow(l)_{r}^{k}\bigr\},\\
&& R_{k}(l)=\min\bigr\{n_{0}:\mathrm{for}\,\,\, n\geq n_{0}\,\,\, n\rightarrow(l)^{k}\bigr\}.
\end{eqnarray*}
So in terms of the above notation we have $s_{n}=R_{3}(5;3^{n+1})$ and $t_{n+1}=R_{5}(t_{n},7)$. Now we state Erd\"{o}s and Rado's upper bound for $R_{k}(l;r)$. For this we need to define a binary operation $\ast$ on positive integers as follows:
\[
a\ast b=a^{b}.
\]
Also for $n\geq3$ we put
\[
a_{1}\ast a_{2} \ast\dots\ast a_{n}=a_{1}\ast(a_{2}\ast(\dots\ast(a_{n-1}\ast a_{n})\dots)).
\]
Then if $1\leq m<n$,
\[
a_{1}\ast a_{2} \ast\dots\ast a_{m}\ast(a_{m+1}\ast\dots\ast a_{n})=a_{1}\ast a_{2} \ast\dots\ast a_{n}.
\]
Erd\"{o}s and Rado proved in \cite{er}:
\begin{theorem}[Erd\"{o}s and Rado] For $r\geq2$ and $l\geq k\geq 2$ we have
\[
R_{k}(l;r)\leq r\ast r^{k-1}\ast r^{k-2}\ast \dots \ast r^{2}\ast[r(l-k)+1].
\]
\end{theorem}
Therefore
\begin{equation}\label{upperbound}
s_{n}\leq 3^{n+1}\ast3^{2n+1}\ast(2.3^{n+1}+1).
\end{equation}

On the other hand by using the stepping-up lemma \cite{grs} of Erd\"{o}s and Hajnal we get a lower bound for $t_{n}$. In fact what we want is an off-diagonal version of the stepping-up lemma which can be obtained by easily modifying the proof of the original version mentioned in \cite{grs}. Thus we have
\begin{theorem}[Off-diagonal Stepping-up Lemma] Suppose $k\geq3$ and $n\nrightarrow(l_{1},l_{2})^{k}$, then $2^{n}\nrightarrow(2l_{1}+k-4,2l_{2}+k-4)^{k+1}$.
\end{theorem}
By a simple coloring argument we can show that $R_{3}(l,4)>2l$, then the stepping-up lemma implies that $R_{5}(2l-1,7)>2^{2l}$. So we have $t_{n+1}=R_{5}(t_{n},7)>2^{t_{n}}$. Hence for $n>1$
\begin{equation}\label{lowerbound}
t_{n}>\mathrm{Tower}_{n}(7)
\end{equation}
where Tower$_{1}(x)=x$, Tower$_{n+1}(x)=2^{\mathrm{Tower}_{n}(x)}$. Now an easy computation through inequalities (\ref{upperbound}), (\ref{lowerbound}) would show that $t_{n}$ exceeds $s_{n}$ when $n$ is large enough.

\section{More on Set Mappings on 4-tuples}
In this section we deal with a set mappings $f\colon[\lambda]^{4}\rightarrow[\lambda]^{<\mu}$ with the additional property:
\begin{equation}\label{middle}
\forall x_{0}<x_{1}<x_{2}<x_{3}\in \lambda\,\,\,[f\{x_{0},x_{1},x_{2},x_{3}\}\cap(x_{1},x_{2})=\emptyset].
\end{equation}
As mentioned in the introduction the above property excludes both the set mappings in Section \ref{mainsection} and in the Komjath-Shelah paper. By slightly modifying Hajnal's proof for Theorem B (see the introduction) at some points we show that
\begin{theorem}
Let $\lambda\geq\mu^{+3}$ and the set mapping $f\colon[\lambda]^{4}\rightarrow[\lambda]^{<\mu}$ satisfies (\ref{middle}). Then $f$ has arbitrary large finite free sets.
\end{theorem}
\begin{proof} We assume that the reader is familiar with Hajnal's proof from (\cite{ehmr}, Section 46) and for his convenience we follow the same terminology. We construct the sets $A_{m}=\{x_{m\xi}:\xi<\mu^{+}\}$ $(m<\omega)$ where $x_{m\xi}$ are pairwise distinct elements of $\lambda$ and a sequence $H_{0}\supset H_{1}\supset\dots\supset H_{m}\supset\dots$$(m\in\omega)$ with the following properties:\\
\begin{itemize}
\item[(i)]$|H_{m}|=\lambda$, $\{x_{m\xi}:\xi<\mu^{+}\}\subset H_{m},$\\
\item[(ii)]$A_{m}<A_{m+1}$,\\
\item[(iii)]$u\notin f\{x_{l_{1}\xi_{1}},x_{l_{2}\xi_{2}},x_{m\eta},v\}$ holds whenever $l_{1}<l_{2}<m<\omega$ and $\xi_{1},\xi_{2},\eta<\mu^{+}$ and $u,v\in H_{m+1}$.\\
\end{itemize}
This can be done exactly as in Hajnal's proof. The extra condition (ii) can easily be fixed by putting $A_{m}$ as the set of the first $\mu^{+}$ elements of $H_{m}$. After doing this we construct an $n$-element set $\{x_{i}:i<n\}$ that is free with respect to $f$ such that $x_{i}\in A_{i}$. Simultaneously with the construction of this set, we also construct sequences of sets $E_{j}^{i}\subset A_{j}$ for $i\leq n-3$ and $j\leq i$ and we would have $x_{i}\in E^{i}_{i}$ for $i\leq n-3$. For $i=n-1,n-2,n-3$ pick $x_{i}\in A_{i}$ arbitrarily, and write $E^{n-3}_{j}=A_{j}$ for $j\leq n-3$. Given $m<n-3$, if $x_{i}$ and $E^{i}_{j}$ has already been defined in case $m<i<n$, $j\leq i$, then define the set mapping $h_{m}$ on $E^{m+1}=\bigcup_{j\leq m}E_{j}^{m+1}$ by putting
\[
h_{m}(u)=\displaystyle{\bigcup_{m+1<i_{1}<i_{2}<n}}f\{u,x_{m+1},x_{i_{1}},x_{i_{2}}\}.
\]
Clearly $|h_{m}(u)|<\mu$ and so by using a lemma quoted in the proof of Hajnal (see \cite{ehmr}, Section 46), there is a set $X_{m}$ free with respect to $h$ such that $X_{m}\cup E^{m+1}_{j}$ has cardinality $\mu$ for each $j\leq m$. Put
\[
E^{m}_{j}=X_{m}\cup E^{m+1}_{j}
\]
for $j\leq m$. Pick an arbitrary element of $E_{m}^{m}$ as $x_{m}$. This finishes the construction.
Now we see that this set is free with respect to $f$. By (ii) we have $x_{0}<x_{1}<\dots<x_{n-1}$. Let
\[
\{x_{i_{0}},x_{i_{1}},x_{i_{2}},x_{i_{3}},x_{i_{4}}\}\subset\{x_{i}:i< n\}
\]
where $i_{0}<i_{1}<i_{2}<i_{3}<i_{4}$. Since $x_{i_{3}},x_{i_{4}}\in H_{i_{2}+1}$, it follows from (i) and (iii) that
\[
x_{i_{3}}\notin f\{x_{i_{0}},x_{i_{1}},x_{i_{2}},x_{i_{4}}\},\,\,\,x_{i_{4}}\notin f\{x_{i_{0}},x_{i_{1}},x_{i_{2}},x_{i_{3}}\}.
\]
Moreover we have
\[
x_{i_{0}}\notin f\{x_{i_{1}},x_{i_{2}},x_{i_{3}},x_{i_{4}}\},\,\,\,x_{i_{1}}\notin f\{x_{i_{0}},x_{i_{2}},x_{i_{3}},x_{i_{4}}\},
\]
since $x_{i_{1}}, x_{i_{2}}\in E^{i_{2}-1}$ and this latter is free with respect to $h_{i_{2}}$. Finally
\[
x_{i_{2}}\notin f\{x_{i_{0}},x_{i_{1}},x_{i_{3}},x_{i_{4}}\}
\]
follows from $x_{i_{0}}<x_{i_{1}}<x_{i_{2}}<x_{i_{3}}<x_{i_{4}}$ and the condition (\ref{middle}) imposed on $f$. Now the proof is complete.
\end{proof}

\section{Set Mappings on Triples}

Komjath-Shelah's second construction in \cite{shelah645} implies that for infinite cardinals $\lambda$, $\mu$, where $\mu$ is regular, $\lambda=\mu^{+n}$, $n\in\omega$ and GCH holds for $\mu^{+l} (l<n)$, there is a cardinal preserving generic extension in which there exists a set mapping $f\colon[\lambda]^{3}\rightarrow[\lambda]^{<\mu}$ with no free set of size $\omega$. The set mapping in question essentially has the following property:
\[
\forall x_{0},x_{1},x_{2}\in \lambda \bigr{[}x_{0}<x_{1}<x_{2}\rightarrow f\{x_{0},x_{1},x_{2}\}\subset x_{0}\bigr{]}.
\]

Also it is easy to see that the proofs of Theorems 4.1 and 4.2 in Hajnal-M\'{a}t\'{e} paper \cite{hajnalmate} for set mappings on pairs can be adapted for set mappings on triples and then we have the following theorem.

\begin{theorem}\label{hajnalmaterestricted} Let $\mu,\lambda$ be infinite cardinals and consider the set mapping $f\colon[\lambda]^{3}\rightarrow[\lambda]^{<\mu}$,
\begin{itemize}
\item[(i)] if $\lambda$ is regular, $\lambda=\mu$ and for all $x_{0}<x_{1}<x_{2}\in\lambda$, $f\{x_{0},x_{1},x_{2}\}\subset (x_{2},+\infty)$, then $f$ has a free set of cardinality $\lambda$.
\item[(ii)] if $\lambda=\mu^{+n}(n\geq 1)$ and  for all $x_{0}<x_{1}<x_{2}\in\lambda$, $f\{x_{0},x_{1},x_{2}\}\subset (x_{1},x_{2})$, then $f$ has a free set of size $\omega$.
\end{itemize}
\end{theorem}

It is interesting to note that recently Komjath in \cite{komjath-preprint} has proved several results of these kinds for a general set mapping on $k$-tuples.

Now it remains to see what happens if a set mapping $f\colon[\lambda]^{3}\rightarrow[\lambda]^{<\mu}$ satisfies the following property:
\[
\forall x_{0}<x_{1}<x_{2}\in\lambda\,\, \bigr{[}f\{x_{0},x_{1},x_{2}\}\subset (x_{0},x_{1})\bigr{]}.
\]

The following theorem can be proved by an adaptation of Komjath-Shelah's second construction. Here we show that our forcing construction in Section \ref{mainsection} will enable us to give another proof for this case. In fact we show that it is consistent that the above $f$ has no free set of size $\omega$. In precise terms:

\begin{theorem}\label{k=3}
Assume that $n<\omega, k=3$ and $\lambda=\mu^{+n}$ for some regular cardinal $\mu$. Let $\mathbb{P}_{n}=\mathbb{P}^{\mu,\lambda}_{n}$ be as in Theorem \ref{main}. Then in $V^{\mathbb{P}_{n}}$ there is a set mapping $F\colon[\lambda]^{3}\rightarrow[\lambda]^{<\mu}$ such that $F$ has no infinite free set and for all $x_{0}<x_{1}<x_{2}\in\lambda$ we have $F\{x_{0},x_{1},x_{2}\}\subset(x_{0},x_{1})$.
\end{theorem}
\begin{proof}
By Theorem \ref{main}, in $V^{\mathbb{P}_{n}}$ there are functions $F\colon[\lambda]^{3}\rightarrow[\lambda]^{<\mu}$ and
\[
\rho_{0}\colon[\lambda]^{2}\rightarrow\mu^{+n},\,\dots\,,\,\rho_{n}\colon[\lambda]^{2}\rightarrow\mu
\]
such that $F$ is $3$-generated by $\Gamma=\{\rho_{0},\dots,\rho_{n}\}$. We show that $F$ has no infinite free set. By way of contradiction suppose not and let $A\subset\lambda$, $|A|=\omega$, be a free set for $F$. For every $\{\alpha,\beta,\gamma\}\in [A]^{3}$ with $\alpha<\beta<\gamma$ and every $\rho\in\{\rho_{0},\dots,\rho_{n}\}$ we have two possibilities:
\begin{center}
(Ia) $\rho\{\alpha,\beta\}\leq\rho\{\beta,\gamma\}$,\,\,\,\,\,\,(Ib) $\rho\{\alpha,\beta\}>\rho\{\beta,\gamma\}$.
\end{center}
According to which one of the above possibilities occurs we say that the 4-tuple $\langle\rho;\alpha,\beta,\gamma\rangle$ satisfies that possibility. Now we define an equivalence relation on $[A]^{3}$. For $\alpha<\beta<\gamma, \alpha^{'}<\beta^{'}<\gamma^{'}$ in $A$, we put $\{\alpha\beta\gamma\}\sim\{\alpha^{'}\beta^{'}\gamma^{'}\}$ when for each $\rho\in\Gamma$, either both of $\langle\rho;\alpha,\beta,\gamma\rangle$, $\langle\rho;\alpha^{'},\beta^{'},\gamma^{'}\rangle$ satisfy (Ia) or both of them satisfy (Ib). By the Ramsey theorem $\omega\rightarrow(\omega)^{3}$ there exists an infinite subset
\[
B=\{\beta_{0},\beta_{1},\beta_{2}\dots\}\subset A,\,\,\beta_{0}<\beta_{1}<\beta_{2}<\dots
\]
such that either all of $\langle\rho;\alpha,\beta,\gamma\rangle$ in question satisfy (Ia) or all of them satisfy (Ib). We claim this cannot be (Ib), since in this case we would have
\[
\rho\{\beta_{0},\beta_{1}\}>\rho\{\beta_{1},\beta_{2}\}>\dots
\]
which is an infinite sequence of strictly decreasing ordinals.

Again for every $\{\alpha,\beta,\gamma\}\in [B]^{3}$ with $\alpha<\beta<\gamma$ and every $\rho\in\{\rho_{0},\dots,\rho_{n}\}$ we have two possibilities:
\begin{center}
(IIa) $\rho\{\alpha,\beta\}\leq\rho\{\alpha,\gamma\}$,\,\,\,\,\,\,(IIb) $\rho\{\alpha,\beta\}>\rho\{\alpha,\gamma\}$.
\end{center}
We define an equivalence relation on $[B]^{3}$. For $\alpha<\beta<\gamma, \alpha^{'}<\beta^{'}<\gamma^{'}$ in $B$, we put $\{\alpha\beta\gamma\}\sim\{\alpha^{'}\beta^{'}\gamma^{'}\}$ when for each $\rho\in\Gamma$, either both of $\langle\rho;\alpha,\beta,\gamma\rangle$, $\langle\rho;\alpha^{'},\beta^{'},\gamma^{'}\rangle$ satisfy (IIa) or both of them satisfy (IIb). As in the previous case by the Ramsey theorem $\omega\rightarrow(\omega)^{3}$  we deduce that there is an infinite subset
\[
C=\{\gamma_{0},\gamma_{1},\gamma_{2}\dots\}\subset B,\,\,\gamma_{0}<\gamma_{1}<\gamma_{2}<\dots
\]
such that either all of $\langle\rho;\alpha,\beta,\gamma\rangle$ in question satisfy (IIa) or all of them satisfy (IIb). This cannot be (IIb), since we would have the following infinite sequence of strictly decreasing ordinals
\[
\rho\{\gamma_{0},\gamma_{1}\}>\rho\{\gamma_{1},\gamma_{2}\}>\dots.
\]
So far we have obtained that for every $\{\alpha\beta\gamma\}\in [C]^{3}$ and every $\rho\in\Gamma$ both of the following two items occur:
\begin{center}
(Ia) $\rho\{\alpha,\beta\}\leq\rho\{\beta,\gamma\}$,\,\,\,\,\,\,(IIa) $\rho\{\alpha,\beta\}\leq\rho\{\alpha,\gamma\}$.
\end{center}

Now for every $\{\alpha,\beta,\gamma\}\in [C]^{3}$ with $\alpha<\beta<\gamma$ and every $\rho\in\{\rho_{0},\dots,\rho_{n}\}$ we have two possibilities:
\begin{center}
(IIIa) $\rho\{\beta,\gamma\}>\rho\{\alpha,\gamma\}$,\,\,\,\,\,\,(IIIb) $\rho\{\beta,\gamma\}\leq\rho\{\alpha,\gamma\}$.
\end{center}
As in the two previous cases we define an equivalence relation on $[C]^{3}$. For $\alpha<\beta<\gamma, \alpha^{'}<\beta^{'}<\gamma^{'}$ in $C$, we put $\{\alpha\beta\gamma\}\sim\{\alpha^{'}\beta^{'}\gamma^{'}\}$ when for each $\rho\in\Gamma$, both of $\langle\rho;\alpha,\beta,\gamma\rangle$ and $\langle\rho;\alpha^{'},\beta^{'},\gamma^{'}\rangle$ satisfy the same possibility. Using the Ramsey theorem $\omega\rightarrow(\omega)^{3}$ for the third time we obtain an infinite homogenous subset $D\subset C$. Now let $\{\alpha,\beta,\gamma,\gamma^{'}\}\subset D$, $\alpha<\beta<\gamma<\gamma^{'}$. We show that for every $\rho\in\Gamma$, $\beta$ is $\rho$-close to $\{\alpha,\gamma,\gamma^{'}\}$. Fix an arbitrary $\rho\in\Gamma$. By (Ia) we have $\rho\{\beta,\gamma\}\leq\rho\{\gamma,\gamma^{'}\}$. Observe that if (IIIa) occurs, then $\rho\{\beta,\gamma^{'}\}<\rho\{\gamma,\gamma^{'}\}$ and if (IIIb) occurs, then $\rho\{\beta,\gamma^{'}\}\leq\rho\{\alpha,\gamma^{'}\}$. This shows that in both cases of (IIIa) and (IIIb), $\beta$ is $\rho$-close to $\{\alpha,\gamma,\gamma^{'}\}$. Since $\rho$ was arbitrary, we conclude that $\beta\in F\{\alpha,\gamma,\gamma^{'}\}$, which violates the freeness of $A$, a contradiction.
\end{proof}

\subsection*{Acknowledgment} We would like to thank Peter Komjath for letting us know about his paper and also the papers of Gillibert and Wehrung. We would also like to thank the referee for his interesting comment.

\bibliography{reference}
\bibliographystyle{plain}
\end{document}